\documentclass [11pt]{amsart}
\usepackage{amsmath}
\usepackage{amssymb}
\usepackage{comment}
\newtheorem{thm}{Theorem}[section]

\newtheorem{lem}[thm]{Lemma}

\theoremstyle{definition}

\theoremstyle{remark}
\newtheorem{rem}[thm]{Remark}
\numberwithin{equation}{section}

\newcommand{\beas}{\begin{eqnarray*}}
\newcommand{\eeas}{\end{eqnarray*}}
\newcommand{\bes} {\begin{equation*}}
\newcommand{\ees} {\end{equation*}}
\newcommand{\be} {\begin{equation}}
\newcommand{\ee} {\end{equation}}
\newcommand{\bea} {\begin{eqnarray}}
\newcommand{\eea} {\end{eqnarray}}
\newcommand{\ra} {\rightarrow}
\newcommand{\txt} {\textmd}

\newcommand{\R} {\mathbb{R}}
\newcommand{\C} {\mathbb{C}}
\newcommand{\N}{\mathbb{N}}

\newcommand{\T}{\mathbb{T}}
\begin{document}

\title[A local Levinson theorem for compact symmetric spaces] {A local Levinson theorem for compact symmetric spaces}

\author{Mithun Bhowmik}

\address{Department of Mathematics, Indian Institute of Technology, Bombay; India.}

\email{mithun@math.iitb.ac.in, mithunbhowmik123@gmail.com}

\thanks{The author is supported by the Post-Doctoral Fellowship from Indian Institute of Technology, Bombay; India.}

%%% ----------------------------------------------------------------------

\begin{abstract}
A classical result due to Levinson characterizes the existence of non-zero functions defined on a circle vanishing on an open subset of the circle in terms of the pointwise decay of their Fourier coefficients \cite{L2}. We prove certain analogue of this result on compact symmetric spaces.
\end{abstract}



%%% ----------------------------------------------------------------------
\maketitle
%%% ----------------------------------------------------------------------

\section{Introduction}
In this article, we will concern ourselves with the classical problem of determining the re-
lationship between the speed of decay of the Fourier coefficients of an integrable function at
infinity and the size of support of the function in the context of a Riemannian symmetric space of the compact type. To understand the context better let
us consider the following well known statement: 
suppose $f$ is an integrable function on the circle group and the Fourier coefficients $\mathcal F f(m)$ of $f$  
satisfies the following condition
\be\label{introdecay}
|\mathcal F f(m)| \leq C e^{-\delta |m|}, \:\:\:\ \txt{ for all } m \leq  -1,
\ee  
for some positive number $\delta$. If $\mathcal F f(m)\in l^1(\mathbb Z)$, then $f$ is almost everywhere equals to the boundary function of
\bes
g(z)= \sum_{m\in \mathbb Z} \mathcal F f(m) ~ z^m,
\ees
which is clearly analytic in the annulus $A_\delta=\{z\in \C : e^{-\delta} < |z| < 1\}$, and extends continuously to the unit circle. Therefore, $f$ 
cannot be equivalent to zero over any interval of positive length without being entirely
equivalent to zero. This simple observation motivates people to endeavour for a more optimal decay of the Fourier transform for such a conclusion. For instance, suppose that we have a function $f$ whose Fourier coefficients $\mathcal F f(m)$ satisfies the estimate
\bes
|\mathcal F f(m)|\leq Ce^{\frac{-|m|}{\log(|m|)}}, \:\:\:\: \txt{ for all } m \leq -1,
\ees
which is a slower decay compare to (\ref{introdecay}).
Then $f$ may no longer be the boundary function of an analytic function. The question now arises whether or not $f$ has the property of vanishing completely if it vanishes over any interval of positive length. The first result of this kind due to Levinson is the following.
\begin{thm}[\cite{L2}, Theorem I]. \label{introlev}
Let $\psi:\N \ra [0, \infty)$ be a non-decreasing function such that 
\bes
\sum_{m\in \N} \frac{\psi(m)} {m^2} = \infty.
\ees
Suppose $f\in L^1(-\pi, \pi)$ is such that for all $m \in \mathbb Z$ with $m\leq  -1$
\be \label{levdecay}
|\mathcal F f(m)| \leq Ce^{- \psi(|m|)},
\ee
for some positive $C$. Then $f$ cannot be equivalent to zero on any interval of positive length without being equivalent to zero over $(-\pi, \pi)$. 
\end{thm}
Levinson go over to Fourier transforms of functions defined on the real line instead of Fourier series for the functions on circle and have exactly the same result there. These results were refined by Beurling (\cite{K}, Ch. VII, \S B). He proved that under the conditions above $f$ cannot even vanish on a set of positive measure unless $f$ is equvalent to zero. There is a whole body of literature  \cite{HJ, I, K, PW, Sh} devoted to the study of the trade off between the nature of the set on which a function vanishes and the  allowable decay of its Fourier transform. Most of these results deal with functions defined on the circle or on the real line. Very recently, we have extended few results of this genre in the context of higher dimenional Euclidean spaces and on certain classes of noncommutative Lie groups and the corrsponding homogeneous spaces \cite{BRS, BR, BS1}. For the $d$-dimensional  torus $\T^d$, $d \geq 1$, we have the following generalization of Theorem \ref{introlev}. If $f\in L^1(\T^d)$ we then define its Fourier coefficient $\mathcal Ff(m)$ by the formula
\be \label{ftdefn}
\mathcal F f(m)=\int_{\T^d}f(x)e^{- im \cdot x}dx, \:\:\:\: \txt{ for } m\in\mathbb Z^d.
\ee
\begin{thm}[\cite{BRS}, Theorem 3.3]\label{levtorusthm}
Let $\psi: [0,\infty) \to [0,\infty)$ be a non-decreasing function and we set
\bes
S=\sum_{n\in \N} \frac{\psi(n)}{n^2}.
\ees
\begin{enumerate}
\item[(a)] Let $f\in L^1(\T^d)$ be such that 
\be \label{levdecay-torus}
|\mathcal F f(m)| \leq C e^{-\psi(\|m\|)}, \:\:\:\: \txt{ for all } m \in \mathbb Z^d.
\ee
If $f$ vanishes on any nonempty open set $U\subset \T^d$ and $S$ is infinte then $f$ is the zero function.
\item[(b)] If $S$ is finite then there exists a non-zero $f \in L^1(\T^d)$ supported on any given nonempty open set $U \subset \T^d$ satisfying (\ref{levdecay-torus}).  
\end{enumerate}
\end{thm}
It is therefore natural to explore the possibility of extending Theorem \ref{introlev} to more general Riemannian manifolds. Recently, we have obtained an anlogue of Theorem \ref{introlev} for Riemannian symmetric spaces of non-compact type \cite{BR}. Our aim in this paper is to do the same for Riemannian symmetric spaces of compact type which is dual of noncompact symmetric spaces.  The following is our main result in this paper. We  refer the reader to section 3 for the meaning of the symbols.  
\begin{thm}\label{thmsym}
Let $\psi: [0,\infty) \to [0,\infty)$ be a non-decreasing function and we set
\bes
S=\sum_{n\in \N} \frac{\psi(n)}{n^2}.
\ees
There exists $R>0$ such that the following holds for each $0<r<R$:
\begin{enumerate}
\item[(a)] Suppose $f\in C^\infty(U/K)$ and its support is contained in $Exp (\overline{\mathcal B(o, r)})$. Let the Fourier transform $\widetilde f$ of $f$ satisfies the following estimate 
\be \label{levdecay}
|\widetilde f(s\mu, k)| \leq C e^{-\psi(\|\mu\|)}, \:\:\:\: \txt{ for } \mu \in \Lambda^+(U/K), k\in K,  s\in W.
\ee
If the series $S$ is infinite then $f$ is the zero function.
\item[(b)] If $S$ is finite then there exists a non-zero $f \in C^\infty(U//K)$ supported on any given nonempty open set $V\subset Exp(\overline{\mathcal B(o, r)})$ satisfying the estimate
\be \label{levdecayconv}
|\widehat f(\mu)|\leq Ce^{-\psi(\|\mu\|)},
\ee 
for all $\mu\in \Lambda^+(U/K)$.
\end{enumerate}
\end{thm}

\begin{rem}
\begin{enumerate}
\item[i)] We put restriction on the function $f$ to be supported in a ball around the origin. In this sense Theorem \ref{thmsym} is a local Levinson's theorem for Riemannian symmetric space of the compact type.

\item[ii)] We note that in Theorem \ref{introlev} the decay of the Fourier transform was assumed only in one direction, that is around negative infinity. But in Theorem \ref{thmsym}, (i) the decay of the Fourier transform is uniform in all directions. It is not clear to us whether it is possible to prove an analogue of Theorem \ref{introlev} by assuming the decay of Fourier transform only in {\it some directions}. We refer the reader to \cite{Sh}, Theorem A$'$, where an analogous issue has been addressed for the Euclidean spaces $\R^d$. 
\end{enumerate}
\end{rem}

\section{Complex analytic Theorems}
In this section, we shall prove two results of complex analysis which will be useful throughout. The first one is the following.
\begin{thm} \label{mainlem}
Let $\psi: [0,\infty) \mapsto [0, \infty)$ be a non-decreasing function. Suppose $f$ is an entire function on $\C^d,~ d\geq 1$ and satisfies the following properties: 
\begin{enumerate}
\item[i)] There exists a constant $A$ with $0 < A < \pi$ such that 
\bes 
|f(z)| \leq Ce^{A \|z\|}, \:\:\:\: \txt{ for all } z\in \C^d.
\ees 
\item[ii)]  Restriction of $f$ to $\mathbb Z^d$   satisfies the following estimate
\be \label{mainlemdecay}
|f(m)| \leq Ce^{-\psi(\|m\|)}, \:\:\:\:  \txt{ for all } m\in \mathbb Z^d.
\ee 
\end{enumerate}
If the function $\psi$ satisfies the condition
\be \label{intcond}
\sum_{n\in \N} \frac{\psi(n)}{n^2}=\infty,
\ee
then $f$ vanishes identically on $\C^d$.
\end{thm}

\begin{rem}
\begin{enumerate}
\item[i)] For $d=1$, the theorem above can be thought of as a generalization of Carlson's theorem (see \cite{B}, P.153); and therefore, it is of independent interest.

\item[ii)] The assumption that $A< \pi$ is crucial. The function $f(z) = \sin \pi z$ is of exponential type $\pi$ and vanishes on $\mathbb Z$.
\end{enumerate}
\end{rem}
The main idea behind the proof of the theorem above is to use Theorem \ref{levtorusthm}. Precisely, we shall construct a function $F$ in $L^1(\T^d)$ out of the given function $f$ such that $F$ satisfies the hypothesis (\ref{levdecay-torus}) of Theorem \ref{levtorusthm} and vanishes on a nonempty open set of $\T^d$. This can be done using the Paley-Wiener theorem. We shall first prove the following result.
\begin{lem} \label{cartlem}
Suppose $f$ is an entire function on $\C^d$ and satisfies the estimate
\bes
|f(z)| \leq Ce^{A \|z\|}, \:\: \txt{ for all } z\in \C^d,
\ees 
for some positive real number $A$ less than $\pi$. If $f$ is bounded on $\mathbb Z^d$, then it is bounded on whole  $\R^d$.
\end{lem}
\begin{proof}
For $d=1$, the result is due to M. L. Cartwright \cite{BU}. Here we shall prove the result for $d>1$ by reducing it to the one dimensional case. We fix $l'\in \mathbb Z^{d-1}$ and define 
\bes
F_1(z_1) = f(z_1, l'), \:\:\:\: \txt{ for } z_1 \in \C.
\ees
Clearly, $F_1$ is an entire function on $\C$ of exponential type less than or equals to $A$, which is further less than $\pi$. Also, by hypothesis $F_1$ is bounded on the set of integers and hence by the case $d=1$, it is bounded on $\R$. We now fix $x_1\in \R, l''\in \mathbb Z^{d-2}$ and define the function 
\bes
F_2(z_2)= f(x_1, z_2, l''), \:\:\:\: \txt{ for } z_2\in \C.
\ees
By the same argument $F_2$ is bounded on $\R$ and continuing the process for $d$ times we conclude that $f$ is bounded on $\R^d$. 
\end{proof}

\begin{proof}[Proof of Theorem \ref{mainlem}]
Let us choose $\phi \in C_c^\infty(\R^d)$ with support of $\phi$ is contained in $B(0, 1) \subset \R^d$ such that $\mathcal F \phi(0)=1$. Here $\mathcal F\phi(\xi)$ denotes the Euclidean Fourier transform of $\phi$ defined as in (\ref{ftdefn}) for $\xi$ in $\R^d$ and integrating over $\R^d$. Also $B(0, 1)$ denotes the open ball in $\R^d$ of radius one centred at zero. Clearly, $\mathcal F\phi$ extends as an entire function on $\C^d$. Let $h= \frac{\pi-A}{2}>0$. We now define 
\be \label{gdefn}
g(z) = f(z) ~ \mathcal F \phi(hz), \:\:\:\: \txt{ for } z\in \C^d.
\ee
Clearly, $g$ is an entire function of exponential type less or equal to $A'=A+h< \pi$. By the hypothesis (i), (ii) and Lemma \ref{cartlem}, it follows that $f$ is bounded on $\R^d$. Consequently, from the property of the function $\mathcal F \phi$ and the definition of $g$ we get that for each $N\in \N$ there exists $C_N$ such that
\bes
|g(x)|\leq \frac{C_N}{ (1+\|x\|)^N},\:\:\:\: \txt{ for all } x\in \R^d.
\ees
A standard application of Phragm\'en-Lindel\"off theorem shows that $g$ also satisfies the following estimate (\cite{G}, Lemma 2).
\bes
|g(z)|\leq C_N \frac{e^{A'\|\Im z\|}}{ (1+\|z\|)^N},\:\:\:\: \txt{ for all } z\in \C^d.
\ees 
Hence by the classical Paley-Wiener theorem $g= \mathcal F G$, for some $G\in C_c^\infty(\R^d)$ such that support of $G$ is contained in $B(0, A') \subset \T^d$ (\cite{H2}, Theorem 2.9, p. 15). We now define 
\bes
F(x) = \sum_{m\in \mathbb Z^d} G(x+m), \:\:\:\: \txt{ for all }
x\in \T^d \simeq [-\pi, \pi]^d.
\ees
Then $F\in L^1(\T^n)$ and by Possion summation formula (\cite{SW}, Theorem 2.4. p 251)
\be \label{psformula}
\mathcal FF(m) = \mathcal F G(m) = g(m),\:\:\:\: \txt{ for all }
m\in \mathbb Z^d.
\ee
Since $0< A'< \pi$ and $F$ is supported in the ball $B(0, A')$, therefore $F$ vanishes on an open set in $\T^d$. It now follows from the hypothesis (ii)  and the equality (\ref{psformula}) that 
\bes
|\mathcal F F(m)| = |g(m)|= |f(m)| ~| \mathcal F \phi(hm)| \leq C_{\phi, h} ~ e^{- \psi(\|m\|)}, \:\:\:\: \txt{ for all }
m\in \mathbb Z^d.
\ees
Since $\psi$ is non-decreasing and the series (\ref{intcond}) is divergent, therefore, by Theorem \ref{levtorusthm} applying to the function $F$, we get that $F$ is identically zero. Hence $G$ as well as $g$ vanishes identically on $\R^d$. Since $\mathcal F \phi$ is real analytic function on $\R^d$, therefore non-zero almost everywhere and hence $f$ is zero almost everywhere on $\R^d$.  By continuity, $f$ is zero on $\R^d$ and hence so is on $\C^d$. 
\end{proof}

As a converse of the Theorem \ref{mainlem} we also have the following
\begin{thm} \label{lemconverse}
Let $\psi$ be as in Theorem \ref{mainlem} and the series 
\be\label{seriescovgt}
\sum_{n\in \N}\frac{\psi(n)}{n^2}< \infty.
\ee
Let $b\in \R^d$. Then for each positive number $A$ there exists a non-zero entire function $f_b$ (depending on $b$) on $\C^d$ satisfying the following
\begin{enumerate} 
\item[i)] For each $N\in \N$ there exists a positive constant $C_N$ such that 
\bes
|f_b(z)|\leq C_N\frac{e^{A\|\Im z\|}}{1+\|z\|^N}, \:\:\:\:  \txt{ for all } z\in \C^d. 
\ees
\item[ii)]The restriction of $f_b$ on $\R^d$ satisfies
\bes
|f_b(x)|\leq Ce^{-\psi(\|x\|)}, \:\:\:\: \txt{for } x\in \R^d.
\ees
\item[iii)] For all $\sigma \in S^{d-1}$ and $z\in \C^d$
\bes
f_b\left(\sigma(z+b)-b\right)= f(z).
\ees
Here $S^{d-1}$ denotes the unit sphere in $\R^d$ centred at zero.
\end{enumerate}
\end{thm} 
\begin{proof}
Since the function $\psi$ is non-decreasing it follows from (\ref{seriescovgt}) that
\bes
\int_{1}^{\infty}\frac{\psi(t)}{t^2}~dt < \infty.
\ees
We now define $\psi_0(t)=\psi(2 t)$, for $t\in [0, \infty)$. The above integral is also finite if we replace $\psi$ by $\psi_0$. For $d=1$, a classical result due to Paley and Wiener (see \cite[Chapter IV, \S D, P. 101]{K}) says that for any given $B$ there exists an entire function $g$ on $\C$ such that 
\begin{enumerate}
\item[(i)] For all $\lambda\in \C$ 
\bes 
|g(\lambda)| \leq Ce^{B|\lambda|}.
\ees 
\item[(ii)] For all $t\in \R$
\bes
|g(t)|\leq Ce^{-\psi_0(|t|)}, .
\ees
\end{enumerate}
We can also assume that the function $g$ is even on $\C$. We define an entire function $f$ on $\C^d$ by
\bes
f(z)= f(z_1, \cdots, z_d)= g(\lambda),
\ees
if $\lambda^2= z_1^2 + \cdots + z_d^2$. Since $g$ is even, this is possible.
Clearly, $f$ satisfies properties (i) and (ii) above for $z\in \C^d$ and $x\in \R^d$ respectively. We now define an entire function $f_b$ by
\bes
f_b(z)= \int_{S^{d-1}} f\left(\sigma(z+b)-b\right)~d \sigma, \:\: z\in \C^d.
\ees
Then the function $f_b$ satisfies the following
\begin{enumerate}
\item[i)] For all $z\in \C^d$ 
\bes 
|f_b(z)| \leq Ce^{B\|z\|}.
\ees 
\item[ii)] For all $x\in \R^d$ with $\|x\|\geq 4\|b\|$
\bes
|f_b(x)|\leq Ce^{-\psi_0\left(\|x\|- 2\|b\|\right)}\leq Ce^{-\psi_0(\frac{\|x\|}{2})}= Ce^{-\psi(\|x\|)}.
\ees
Since the function $\psi$ is non-decreasing we can choose the constant $C$ large enough such that the above inequality holds for all $x\in \R^d$. 
\item[iii)] For all $\sigma\in S^{d}$ and $z\in \C^k$
\bes
f_b\left(\sigma(z+b)-b\right)= f_b(z).
\ees
\end{enumerate}
Now, if we choose $B= A/2$ then by (\cite{G}, Lemma 2) $f_b$ also satisfies the estimate that for each $N\in \N$ there exists a positive constant $C_N$ such that 
\bes
|f_b(z)|\leq C_N\frac{e^{A\|\Im z\|}}{1+\|z\|^N}, \:\:\:\:  \txt{ for all } z\in \C^d. 
\ees
This completes the proof.
\end{proof}

\section{Notation and Preliminaries}
In this section, we describe the necessary preliminaries regarding the harmonic analysis on Riemannian symmetric spaces of the compact type. These are standard and can be found, for example, in \cite{H0, H1, H2}. To make the article self-contained, we shall gather only those results which will be used throughout the paper. 

We are considering a Riemannian symmetric space $U/K$, where $U$ is a connected, simply connected, compact, semisimple Lie group and $K$ a closed subgroup with the property that $K=U^\theta$, for an involution $\theta$ of $U$. Here $U^\theta$ denotes the subgroup of $\theta$-fixed points. Since $U$ is simply connected, $U^\theta$ is connected. We will denote the base point in $U/K$ by $o=eK$. Let $\mathfrak u$ denote the Lie algebra of $U$ and $\mathfrak u= \mathfrak k \oplus \mathfrak q$ be the Cartan decomposition associated with the involution $\theta$. Then $\mathfrak k$ is the Lie algebra of $K$ and $\mathfrak q$ can be identified with the tangent space $T_o(U/K)$ at the origin $o$. Let $\langle \cdot, \cdot \rangle$ be the inner product on $\mathfrak u$ defined by $\langle X, Y \rangle =-B(X, Y )$, where $B$ is the Killing form. We assume that the Riemannian metric $g$ of $U/K$ is normalized such that it agrees with $\langle \cdot, \cdot\rangle$ on the tangent space $\mathfrak q = T_o(U/K)$. We denote by $\exp$ the exponential map $\mathfrak u \ra U$, and by $Exp$ the map $\mathfrak q \ra U/K$ given by $Exp(X) = \exp(X)\cdot o$. Let $\mathcal B(0, r)$ be the open ball in $\mathfrak q$ of radius $r> 0$ and centered at $0$ and $\mathcal D(o, r)$ the open metric ball in $U/K$ of radius $r > 0$ and centered at $o$. Similarly $\overline{\mathcal B(0, r)}$ and $\overline{\mathcal D(o, r)}$ stand for the closed balls. The exponential map $Exp$ is surjective and an analytic diffeomorphism $\mathcal B(0, r)\ra \mathcal D(o, r)$ for $r$ sufficiently small.

The inner product on $\mathfrak u$ determines an inner product on the dual space $\mathfrak u^*$ in a canonical fashion. Furthermore, these inner  products have complex bilinear extensions to the complexifications $\mathfrak u_\C$ and $\mathfrak u^*_\C$. All these bilinear forms are denoted by the same symbol $\langle \cdot , \cdot \rangle$.

Let $\mathfrak a \subset \mathfrak q$ be a maximal abelian subspace, $\mathfrak a^*$ its dual space, and $\mathfrak a^*_\C$ the complexified dual space. We assume that $\dim \mathfrak a= d$, called real rank of $U$. Let $\Sigma$ denote the set of non-zero (restricted) roots of $\mathfrak u$ with respect to $\mathfrak a$, then $\Sigma \subset \mathfrak a^*_\C$ and all the elements of $\Sigma$ are purely imaginary on $\mathfrak a$. The multiplicity of a root $\alpha \in \Sigma$ is denoted $m_\alpha$. The corresponding Weyl group, generated by the reflections in the roots, is denoted $W$. We make fixed choice of a positive system $\Sigma^+$ for $\Sigma$ and define $\rho \in i \mathfrak a ^*$ to be the half sum of the roots in $\Sigma^+$, counted with multiplicity. 

Since $U$ is compact there exists a unique (up to isomorphism) connected complex Lie group $U_\C$ with Lie algebra $\mathfrak u_\C$ which contains $U$ as a real Lie subgroup. Let $\mathfrak g$ denote the real form $\mathfrak k \oplus i \mathfrak q$ of $u_\C$ and let $G$ denote the connected real Lie subgroup of $U_\C$ with this Lie algebra. Then $\mathfrak g_\C= \mathfrak u_\C$ as complex vector spaces and $U_\C$ complexifies $G$ as well as $U$. For this reason we shall denote $U_\C$ also by $G_\C$. The Cartan involutions of $\mathfrak u$ and $U$ extend to involutions of $\mathfrak g_\C$ and $G_\C$, which we shall denote again by $\theta$, and which leave $\mathfrak g$ and $G$ invariant. The corresponding Cartan decomposition of $\mathfrak g$ is $\mathfrak g= \mathfrak k \oplus i\mathfrak q$. It follows that $K=G^\theta$ is maximal compact in $G$, and $G/K$ is a Riemannian symmetric space of the non-compact type. 

We denote by $\mathfrak g= \mathfrak k \oplus i\mathfrak a \oplus \mathfrak n$ and $G=KAN$ the Iwasawa decompositions of $\mathfrak g$ and $G$ associated with $\Sigma^+$. Here $A= \exp(i\mathfrak a)$ and $N=\exp \mathfrak n$. Furthermore, we let $H: G\ra i\mathfrak a$ denote the Iwasawa projection given by $
H(k \exp Y n)= Y$, for $k\in K, Y\in i\mathfrak a$ and $n\in N$. Let $K_\C, A_\C$ and $N_\C$ denote the connected subgroups of $G_\C$ with Lie algebras $\mathfrak k_\C, \mathfrak a_\C$ and $\mathfrak n_\C$. Then $G_\C/K_\C$ is a symmetric space and it carries a natural complex structure with respect to which $U/K$ and $G/K$ are totally real submanifolds of maximal dimension. We now have the following result.
\begin{lem} [\cite{OS1}, Lemma 2.1] \label{lempre}
There exists an open $K_\C \times K$-invariant neighborhood $\mathcal V^a$ of the neutral element $e$ in $G_\C$ and a holomorphic map $H: \mathcal V^a \ra \mathfrak a_\C$
which agrees with Iwasawa projection on $\mathcal V^a \cap G$, such that $u\in K_\C \exp(H(u))N_\C$, for all $u\in \mathcal V^a$.
\end{lem}
We call the map $H$ the complexified Iwasawa projection. We shall now recall the local Fourier theory for $U/K$ based on elementary representation theory. An irreducible unitary representation $\pi$ of $U$ is said to be a $K$-spherical representation if there exists a non-zero $K$-fixed vector $e_\pi$ in the representation space $V_\pi$. The vector $e_\pi$ (if it exists) is unique up to multiplication by scalars. A spherical representation $\pi=\pi_\mu$ is labeled by an element $\mu \in \mathfrak a_\C ^*$, which is the restriction, from a compatible maximal torus, of the highest weight of $\pi$ (see \cite{H2}, p. 538). We denote by $\Lambda^+(U/K)\subset \mathfrak a_\C^\ast$ the set of these restricted highest weights, so that $\mu \mapsto \pi_\mu$ sets up a bijection from $\Lambda^+(U/K)$ onto the set of equivalence classes of irreducible $K$-spherical representations. According to Helgason's theorem, every $\mu \in \Lambda^+(U/K)$ satisfies
\bes
\frac{\langle \mu, \alpha \rangle}{\langle  \alpha, \alpha \rangle} \in \mathbb Z^+, \:\:\:\: \txt{ for all } \alpha \in \Sigma^+,
\ees
\cite[Ch IV, Theorem 4.1]{H2}. 
A root $\alpha\in \Sigma$ is said to be unmultipliable if $2\alpha \notin \Sigma$.  
Let $\{\beta_1,\cdots, \beta_d\}$ be a basis  consisting of simple unmultipliable roots. We now define $w_1, \cdots, w_d \in \mathfrak a_\C^\ast$
by the conditions
\be \label{wdefn}
\frac{\langle w_j, \beta_k\rangle}{\langle \beta_k, \beta_k \rangle}= \delta_{jk}.
\ee
Let $\mu\in \mathfrak a^\ast$. Then $\mu \in \Lambda^+(U/K)$ if any only if
\bes
\mu= \sum_{j=1}^d \mu_j w_j, \:\:\:\: \txt{ with } \mu_j\in \mathbb Z^+, j\in\{1, \cdots, d\},
\ees
\cite[Ch II, Proposition 4.23]{H1}. The weights $w_1, \cdots, w_d$ are called fundamental weights. Let $\Lambda(U/K)=\sum_{j=1}^d \mathbb Z w_j$. We have then
\be \label{Lambdareln}
\Lambda^+(U/K)= \Lambda(U/K)/W,
\ee
\cite[Ch V]{H2}. We also need the following fact that there exists two positive numbers $c_1$ and $c_2$ such that 
\be \label{normeqv}
c_1\|n\|\leq \|(n_1w_1 +\cdots +n_dw_d)\|\leq c_2\|n\|,
\ee
for all $n=(n_1, \cdots, n_d)\in \mathbb Z^d$.
For each $\mu\in \Lambda^+(U/K)$ we fix an irreducible unitary spherical representation $(\pi_\mu, V_\mu )$ of $U$ and a unit $K$-fixed vector $e_\mu \in V_\mu$. The spherical function on $U/K$ associated with $\mu$ is the matrix coefficient
\be \label{psidefn}
\psi_\mu(u) = \langle \pi_\mu(u)e_\mu, e_\mu \rangle, \:\: u\in U,
\ee
viewed as a function on $U/K$. It is $K$-biinvariant, that is, $K$-invariants on both sides as a function on $U$, and it is independent of the choice of the unit vector $e_\mu$. Henceforth, we shall denote the set of $K$-biinvariant functions in $L^1(G/K)$ by $L^1(G//K)$. The spherical Fourier transform of a continuous $K$-biinvariant function $f$ on $U$ is the function $\widehat f$ on $\Lambda^+(U/K)$ defined by
\be \label{defnsft}
\widehat f(\mu) = \int_{U/K} f(u) ~ \overline{\psi_\mu(u)} ~ du,
\ee
where $du$ is the Riemannian measure on $U/K$, normalized with total measure $1$. The spherical Fourier series for $f$ is the series given by
\bes
\sum_{\mu \in \Lambda^+(U/K)}  d(\mu) ~ \widehat f(\mu) ~ \psi_\mu,
\ees
where $d(\mu)= \dim V_\mu$. The Fourier series converges to $f$ in $L^2$ and, if $f$ is smooth, absolutely and uniformly (see \cite{H2}, P. 538). 

We shall now recall the local Paley-Wiener theorem for $K$-biinvariant functions on compact symmetric space due to \'Olafsson and Schlichtkrull \cite{OS}. Let $C_r^\infty(U//K)$ denote the space of $K$-biinvariant smooth functions on $U$ supported in $Exp ~\overline{\mathcal B(0, r)}$. For $r > 0$, let $PW_r(\mathfrak a)$ denote the space of holomorphic
functions $\phi$ on $\mathfrak{a}_{\C}^*$ satisfying the following:
\begin{enumerate}
\item[a)] For each $N\in \N$ there exists a constant $C_N > 0$ such that
\bes
|\phi(\lambda)| \leq C_N(1 +\|\lambda\|)^{-N} e^{r\|\Re\lambda\|},\:\:\:\: \txt{ for all } \lambda\in\mathfrak{a}_{\C}^*.
\ees
\item[b)] For all $w \in W, \lambda \in\mathfrak{a}_{\C}^*$,
\bes
\phi(w(\lambda + \rho)-\rho) = \phi(\lambda), 
\ees 
where 
\bes
\rho = \frac{1}{2} \Sigma_{\alpha \in \Sigma^+ }  ~ m_\alpha \alpha.
\ees
\end{enumerate}

\begin{thm} [\cite{OS}, Theorem 4.2] \label{pwthm} 
There exists $R_0 > 0$ such that the following holds for each $0 < r < R_0$:
\begin{enumerate}
\item[i)] Let $f\in C_r^\infty(U//K)$. Then the spherical Fourier transform $\widehat f : \Lambda^+(U/K)\ra \C$ of $f$ extends to a function in $PW_r(\mathfrak a)$.
\item[ii)]Let $\phi \in PW_r(\mathfrak a)$. There exists a unique function $f \in C_r^\infty(U//K)$ such that $\widehat f(\mu) = \phi(\mu)$, for all $\mu \in \Lambda^+(U/K)$.
\end{enumerate}
\end{thm}

We now define the the Fourier transform of an integrable function $f$ on $U/K$ which is not necessarily $K$-biinvariant. The theory essentially originates from Sherman \cite{S}. For each $\mu\in \Lambda^+(U/K)$, we fix an irreducible unitary spherical representation $(\pi_\mu, V_\mu)$ of $U$ and a unit $K$-fixed vector $e_\mu\in V_\mu$. Furthermore, we fix a highest weight vector $v_\mu$ of weight $\mu$, such that $\langle v_\mu, e_\mu \rangle = 1$.
We now define the Fourier transform of an integrable function $f$ on $U/K$ by
\bes
\tilde f(\mu, k)= \int_{U/K} f(u) \langle \pi_\mu(k)v_\mu, \pi_\mu(u)e_\mu\rangle du,
\ees
(see \cite{OS1}). If $f$ is $K$-invariant, then $\tilde f(\mu, k)$ is independent of $k$. Integration over $K$ then shows that this definition agrees with the spherical Fourier transform in (\ref{defnsft}).

We now invoke the complex group $G_\C$ and the complexified Iwasawa projection
defined Lemma \ref{lempre}. Let $\mathcal V^a \subset G_\C$ and $H : \mathcal V^a\ra \mathfrak a_\C$ be as in Lemma \ref{lempre}, and let $\mu \in \Lambda^+(U/K)$. Since $\pi_\mu$ extends to a holomorphic representation of $G_\C$ and $\langle v_\mu, e_\mu \rangle= 1$, it follows from Lemma \ref{lempre} that 
\bes
\langle \pi_\mu(u)v_\mu, e_\mu \rangle = e^{\mu\left(H(u)\right)}, \:\:\:\: \txt{ for all } u\in \mathcal V^a. 
\ees 
We define 
\bes
\mathcal V= \{u^{-1}: u\in \mathcal V^a\} \subset G_\C.
\ees
Then it follows from the above equality that
\be \label{ftexprn}
\langle \pi_\mu(k)v_\mu, \pi_\mu(u)e_\mu \rangle =e^{\mu\left(H(u^{-1}k)\right)},
\ee 
for $k\in K, u\in U\cap \mathcal V$ and $\mu\in \Lambda^+(U/K)$. On the other hand it  follows from \cite[Theorem 7.1]{OS1} that there exists a number $R_1 > 0$ such that $Exp~ (\overline{\mathcal B(0, R_1)}) \subset U \cap \mathcal V$. Therefore we  have the following
 \begin{lem}[\cite{OS1}, Lemma 3.2]\label{lempre2}
Let $f$ be an integrable function on $U/K$ and supported in $Exp~ (\overline{\mathcal B(0, R_1)})$. Then
\bes
\widetilde f(\mu, k) = \int_{U/K}f(u) e^{\mu\left(H (u^{-1} k) \right)} du,
\ees
for all $k\in K/M$ and $\mu \in \Lambda^+(U/K)$.
\end{lem}
Since all norms are quivalent on $\mathfrak a_\C^\ast$, we get $c^\ast> 1$ such that 
\bes
\left(|\lambda_1|+\cdots+ |\lambda_d|\right) \leq c^\ast \|\lambda\|, \txt{ for all } \lambda=\left(\sum_{j=1}^d\lambda_j w_j\right)\in \mathfrak a_\C^\ast.
\ees
\begin{rem}\label{Rdefn}
Let us now define $R= \frac{\min\{R_0, R_1\}}{c_\ast^2}$. Then $R$ satisfies the following (\cite[Remark 4.3]{OS})
\begin{enumerate}
\item[(i)] $Exp(\overline{\mathcal B(0, r)})$ is diffeomorphism onto its image for all $0< r< R$.
\item[(ii)] $c_\ast^2 R\leq \frac{\pi}{2\|\alpha\|}$, for all $\alpha \in \Sigma$.
\item[(iii)] $c_\ast^2 R\leq \frac{\pi}{\|w_j\|}$, for all fundamental weights $w_j$.
\end{enumerate}
\end{rem}

% Let $\mathfrak{h} \subset \mathfrak{u}$ be a Cartan subalgebra containing $\mathfrak{a}$, then $\mathfrak{h} = \mathfrak{h}_m \oplus \mathfrak{a}$, where $\mathfrak{h}_m = \mathfrak{h} \cap \mathfrak{k}$. Let $\Delta$ denote the set of roots of $\mathfrak{u}$ with respect to $\mathfrak{h}$, then $\Sigma$ is exactly the set of non-zero restrictions to $\mathfrak{a}$ of elements of $\Delta$. We fix a set $\Sigma^+ \subset \Sigma$ of positive restricted roots, and a compatible set $\Delta^+ \subset \Delta$ of positive roots. The set of dominant integral linear functionals on h is 
%\bes
%\Lambda^+(\mathfrak{h}) = \left\{ \lambda \in \mathfrak{h}_\C^* ~\big| \:\:  \frac{2 \langle, \alpha \rangle}{\langle \alpha, \alpha \rangle} \in \Z^+, \: \txt{ for all } \alpha \in \Delta^+ \right\},
%\ees
%where $\Z^+ = \{0, 1, 2, \cdots\}$. We notice that since $\mathfrak{u}$ is compact, all elements of $\Delta$ and $\Lambda^+(\mathfrak{h})$ take purely imaginary values on $\mathfrak{h}$. Let $\Lambda^+(U)\subset \mathfrak{h}^*$ denote the set of highest weights of irreducible representations of $U$, then $ \Lambda^+(U)\subset \Lambda^+(\mathfrak{h})$ with equality if and only if $U$ is simply connected. Let $\Lambda^+_K(U)$ denote the subset of $\Lambda^+(U)$ which corresponds to $K$-spherical representations. Let $\Lambda^+(U/K)$ denote the set of restrictions $\mu = \lambda|_\mathfrak{\mathfrak{a}}$ where $\lambda \in \Lambda^+_K(U )$. This set is in bijective correspondence with $\Lambda^+_K(U)$.

\section{Proof of Theorem \ref{thmsym}} 
We now complete the proof of the main theorem. For $f\in L^1(U/K)$ we first define the $K$-biinvariant component $\mathcal Sf$ of $f$ by the integral
\bes
\mathcal Sf(u) = \int_Kf(ku)~dk, \:\: u\in U/K,
\ees
and for $u \in U$ we define the left translation operator $l_u$ on $L^1(U/K)$ by 
\bes
l_u f(u')= f(uu'), \:\:  u'\in U/K.
\ees
\begin{rem}
Usually one defines the operator $l_u$ as left translation by $u^{-1}$. The reason we have defined $l_u$ as left translation by $u\in G$ because then it follows that ${\mathcal S}(l_uf)={\mathcal S}(l_{u_1}f)$ if $uK=u_1K$.
\end{rem}
We now state the relationships between the Fourier transform of a suitable function $f$ and $l_uf$, for some $u\in U$. Let us fix a positive number $r$ with $r< R$, where $R$ is defined as in Remark \ref{Rdefn}. Suppose $f\in L^1(U/K)$ is such that $\txt{supp } f \subset Exp(\overline{\mathcal B(0, r)})$. Let us consider $\epsilon= (R-r)/2$ and $u_0\in Exp(\overline{\mathcal B(0, \epsilon)})$. It is easy to see that $\txt{supp }l_{u_0}f \subset Exp(\overline{\mathcal B(0, R)})$. By Lemma \ref{lempre2} we get that
\bea \label{ftintreln}
(l_{u_0}f{\widetilde{)}}(\mu, k) &=& \int_{U/K}(l_{u_0}f)(u) e^{\mu\left(H (u^{-1} k) \right)} du \nonumber\\
&=& \int_{U/K}f(u) e^{\mu\left(H (u^{-1} u_0 k) \right)} du,
\eea
We now use the following fact
\be \label{Hreln}
H(u^{-1} u_0k)= H(u^{-1}\kappa(u_0k))+ H(u_0k), \:\:\:\: u\in Exp(\overline{\mathcal B(0, r)}), k\in K, 
\ee
where $\kappa(k \exp(Y)n)= k$ (see \cite{OS}). This follows from the fact that 
$A_\C$ normalizes $N_\C$. From (\ref{ftintreln}) and (\ref{Hreln}) we get that
\be\label{trnsreln}
(l_{u_0}f{\widetilde{)}}(\mu, k) = e^{\mu \left(H(u_0k)\right)} \widetilde f(\mu, \kappa(u_0 k)),
\ee
(see \cite{H1}, Chapter III, \S 2, P. 209 for noncompact case). 
For a nonzero integrable function $f$, its $K$-biinvariant component $\mathcal S(f)$ may be zero. However, the following lemma shows that there always exists $u\in U$ such that $\mathcal S(l_uf)$ is nonzero (see \cite[Lemma 4.6]{BR} for noncompact case). 
\begin{lem} \label{nonzeroradiallem}
If $f\in L^1(U/K)$ is nonzero then for every $\epsilon$ positive there exists $u\in U$ with $uK\in Exp(\mathcal B(0, \epsilon))$ such that $\mathcal S(l_uf)$ is nonzero.
\end{lem}
\begin{proof}
Suppose the result is false. Then there exists a positive $\epsilon$ such that for all $uK\in Exp(\mathcal B(0, \epsilon))$ the function  $\mathcal S(l_uf)$ is zero. Hence, for all $t$ positive we have
\bes
\int_U \mathcal S(l_uf)(v) ~ h_t(v^{-1}) ~ dv =0.
\ees
Here $h_t$ denotes the heat kernel on $U/K$. This implies that $(f*h_t)(uK)$ is zero for all positive number $t$. In fact, 
\beas
\int_U \mathcal S(l_uf)(v)~ h_t(v^{-1})~dv
&=& \int_U \left(\int_K l_uf(kv)~dk\right)~h_t(v^{-1})~ dv\\
&= & \int_U l_uf(v)~h_t(v^{-1})~dv \\
&& \txt{(using change of variable } kv \mapsto v) \\
&=& \int_U f(uv)~ h_t(v^{-1})~ dv\\
&=& f*h_t(uK).
\eeas
It follows that $f*h_t$ vanishes on the open ball $Exp(\mathcal B(0, \epsilon))$, for all $t$ positive. Since the function $f\ast h_t$ is real analytic therefore vanishes identically. Consequently the function $f$ is the zero which contradicts our assumption. 
\end{proof}
We are now in a position to prove our main result.
\begin{proof}
[Proof of Theorem \ref{thmsym}.]
We first prove part (a) under the additional assumption that it is $K$-biinvariant. Therefore we now have $f\in C^\infty(U//K)$ with $\txt{supp } f \subset Exp(\overline{\mathcal B(0, r)})$, for some $0< r< R$ and the spherical transform $\widehat f$ satisfies the following 
\be \label{levdecay1}
|\widehat f(s\mu)| \leq C e^{-\psi(\|\mu\|)}, \:\:\:\: \txt{ for } \mu \in \Lambda^+(U/K), s\in W.
\ee
We will show that if $S$ is infinite then $f$ is the zero function. Since $f\in C_r^\infty(U//K)$ by Theorem \ref{pwthm}, (i), $\widehat f$ is an entire function in $PW_r(\mathfrak{a})$. Let $\{w_1, \cdots, w_d\}$ be the  fundamental weights defined by the relation (\ref{wdefn}). We now define an entire function $F$ on $\C^d$ by
\bes
F(z)= \widehat f (z_1 w_1 + \cdots + z_d w_d), \:\:\:\: \txt{ for all } z=(z_1, \cdots, z_d)\in \C^d.
\ees
Using property $(a)$ of the space $PW_r(\mathfrak a)$ it follows that
\bes
|F(z)|\leq |\widehat f(z_1 w_1+ \cdots + z_dw_d)|\leq C e^{r|z_1|\|w_1\|+ \cdots + r|z_d|\|w_d\|}\leq Ce^{\frac{\pi}{c_\ast}\|z\|},
\ees
for all $z\in \C^d$. We also have form (\ref{levdecay1}), (\ref{Lambdareln}) and (\ref{normeqv}) that
\bes
F(n_1, \cdots, n_d)= \widehat f(n_1 w_1 + \cdots +n_dw_d)=Ce^{-\psi\left(c_1\|n\|\right)}= Ce^{-\psi_{c_1}(\|n\|)}, 
\ees
for all $n\in \mathbb Z^d$, where $\psi_{c_1}$ is the dilation of $\psi$ by $c_1$.
Therefore, by Lemma \ref{mainlem} we get that $F$ is the zero function and hence so is $f$.

We shall now reduce the general case to the case of $K$-biinvariant functions by using the radialization operator $\mathcal S$. The idea is same as we did in \cite{BR} (see step 3 of the proof of Theorem 1.2 in \cite{BR}). If possible, let $f \in C^\infty(U/K)$ be a nonzero function supported on $Exp(\overline{\mathcal B(0, r)})$ and satisfies the estimate (\ref{levdecay}).
An application of Lemma \ref{nonzeroradiallem} for $\epsilon=(R-r)/2$ shows that there exists $u_0K\in Exp(\mathcal B (0,\epsilon))$ such that $\mathcal S(l_{u_0}f)$ is nonzero. 
The spherical Fourier transform of the $K$-biinvariant function $\mathcal S(l_{g_0}f)$ is given by 
\bea 
\widehat{\mathcal S\big(l_{u_0}f\big)}(\mu) & = & \int_{U/K} \mathcal S\big(l_{u_0}f \big)(u)~\overline{\psi_\mu (u)}~ du\nonumber\\
&=& \int_{U/K} \left(\int_K (l_{u_0}f)(ku) ~ dk \right) ~ \overline{\psi_\mu(u)}~ du\nonumber \\
&=& \int_{U/K} f \big(u_0u \big)~ \overline{\psi_\mu(u)}~ du\nonumber\\
&=& \int_{U/K} f(u) ~ \overline{\psi_\mu(u_0^{-1}u)}~du, \label{ftlghash}
\eea
using change of variable $ku \mapsto u$ and $K$-biinvariance of $\psi_\mu$. We shall now use the following fact (\cite[p.7]{OS1})
\be \label{piuexpn}
\pi_\mu(u)e_\mu= \int_{K} e^{-(\mu+2\rho)H(u^{-1}k)}~\pi_\mu(k)v_\mu~ dk, \:\:\:\: \txt{ for } u\in U\cap \mathcal V.
\ee
Using the relation (\ref{ftexprn}) and (\ref{piuexpn}), it follows from the definiton (\ref{psidefn}) of $\psi_\mu$ that for all $u \in Exp(\mathcal B (0, r))$ 
\bea \label{phiprod}
\psi_\mu(u_0^{-1}u) &=& \langle \pi_\mu(u)e_\mu, \pi_\mu(u_0)e_\mu \rangle \nonumber\\
&=& \int_{K} e^{-(\mu+2\rho)H(u^{-1}k)}~\langle \pi_\mu(k)v_\mu, \pi_\mu(u_0)e_\mu\rangle~ dk \nonumber\\
&=& \int_{K} e^{-(\mu+2\rho)H(u^{-1}k)}~e^{\mu\left(H(u_0^{-1}k)\right)}~dk.
\eea 
Therefore, since $\rho$ are purely imaginary on $\mathfrak a$, it follows from (\ref{ftlghash}) and (\ref{phiprod}) that
\beas
\widehat{\mathcal S\big(l_{u_0}f\big)}(\mu)&=&\int_{U/K} \int_K f(u) ~ e^{(-\bar{\mu}+2\rho)H(u^{-1}k)}~e^{\bar{\mu}\left(H(u_0^{-1}k)\right)}~dk~du \\
&=& \int_K \widetilde f(-\bar{\mu} + 2\rho, k) e^{\bar{\mu} \left(H(u_0^{-1}k)\right)} ~dk.
\eeas
Now, using that fact that $H$ is continuous and the hypothesis (\ref{levdecay}) it follows from above that
\bes
|\widehat{\mathcal S\big(l_{u_0}f\big)}(s\mu)| \leq C e^{-\psi(\|s\mu- 2\rho\|)}\leq C_\rho e^{-\psi \left(\frac{\|\mu\|}{2}\right)} \:\:\:\: \txt{ for } \mu \in \Lambda^+(U/K), s\in W
\ees 
Therefore, by the biinvariact case it follows that $\mathcal S(l_{u_0}f)$ is zero which contradicts the fact that $f$ is non-zero.

We shall now prove part (b). Here the series $S$ is convergent. Because of the identification of $\mathfrak a_\C^\ast$ with $\C^d$, it follows from Theorem \ref{lemconverse} that  
for $b\in \mathfrak a^\ast$ and $A>0$ there exists a non-zero entire function $F_b$ on $\mathfrak a_\C^\ast$ satisfying the following
\begin{enumerate} 
\item[i)] For each $N\in \N$ there exists $C_N>0$ such that 
\bes
|F_b(\lambda)|\leq C_N\frac{e^{A\|\Im \lambda\|}}{1+\|\lambda\|^N}, \:\:\: \txt{ for all } \lambda\in \mathfrak a_\C^\ast.
\ees
\item[ii)]The restriction of $F_b$ on $\mathfrak a^\ast$ satisfies
\bes
|F(x)|\leq Ce^{-\psi(\|x\|)}, \:\:\:\: \txt{for } x\in \mathfrak a^\ast,
\ees
\item[iii)] For all $\sigma\in S^{d-1}$ and $\lambda\in \mathfrak a_\C^\ast$
\bes
F_b\left(\sigma(\lambda+ b)-b\right)= F_b(\lambda).
\ees
\end{enumerate}
We now define 
\bes
F(\lambda)= F_{i\rho}(i\lambda), \:\:\:\: \txt{ for } \lambda \in \mathfrak a_\C^\ast.
\ees
It follows from property (i) and (iii) of the function $F_{i\rho}$ that $F\in PW_A(\mathfrak a)$ and consequently, we get from Theorem \ref{pwthm}, (ii) that $f\in C_c(U//K)$ such that 
\bes
\widehat f(\mu)= F(\mu), \:\:\:\: \txt{ for all } \mu\in \Lambda^+(U/K).
\ees
Using property (ii) of the function $F_{i\rho}$ it now follows from the above equation that $\widehat f$ satisfies the estimate (\ref{levdecayconv}).   
\end{proof}

\textbf{Acknowledgement.} We would like to thank Swagato K. Ray for suggesting this problem and for the many useful discussions during the course of this work. We would like to thank Sanjoy Pusti for several useful discussions during the course of this work.

\end{document}